\documentclass[a4paper,11pt, reqno]{amsart}
\usepackage[margin=2cm]{geometry}

\usepackage{natbib}
\usepackage{amssymb,amsthm}
\usepackage{mathtools}
\usepackage[colorlinks=true, allcolors=blue]{hyperref}

\usepackage{amsmath}
\usepackage[colorlinks=true, allcolors=blue]{hyperref}
\usepackage{verbatim}
\usepackage{soul}
\usepackage{siunitx}
\usepackage{booktabs} 
\usepackage{etoolbox} 
\usepackage{multirow}
\usepackage{graphicx}

\usepackage{subfigure}
\usepackage{subcaption}
\usepackage{adjustbox}

\newtheorem{propo}{Proposition}[section]
\newtheorem{theo}{Theorem}[section]

\newtheorem{remark}{Remark}[section]

\def\R{\mathbb{R}}
\def\S{\mathbb{S}}

\def\x{\mathbf{x}}
\def\c{\mathbf{c}}
\def\z{\mathbf{z}}
\def\r{\mathbf{R}}
\def\xxi{\mathbf{\xi}}
\def\aalpha{\mathbf{\alpha}}
\def\bbeta{\mathbf{\beta}}

\def\X{\mathcal{X}}
\def\K{\mathcal{K}}

\allowdisplaybreaks

\robustify\bfseries 

\let\origmaketitle\maketitle
\def\maketitle{
	\begingroup
	\def\uppercasenonmath##1{} 
	\let\MakeUppercase\relax 
	\origmaketitle
	\endgroup
}

\begin{document}

\title[]{\Large A Mathematical Optimization Approach to\\Multisphere Support Vector Data Description}

\author[V. Blanco, I. Espejo, R. P\'aez, \MakeLowercase{and} A.M. Rodr\'iguez-Ch\'ia]{
{\large V\'ictor Blanco$^{\dagger}$, Inmaculada Espejo$^{\ddagger}$, Ra\'ul P\'aez$^{\ddagger}$, and Antonio M. Rodr\'guez-Ch\'ia$^{\ddagger}$}\medskip\\
$^\dagger$Institute of Mathematics (IMAG), Universidad de Granada\\
$^\ddagger$Dpt. of Stats. \& Operations Research, Universidad de C\'adiz\\
\texttt{vblanco@ugr.es}, \texttt{inmaculada.espejo@uca.es}, \texttt{raul.paez@uca.es}, \texttt{antonio.rodriguezchia@uca.es}
}

\maketitle

\begin{abstract}
We present a novel mathematical optimization framework for outlier detection in multimodal datasets, extending Support Vector Data Description approaches. We provide a  primal formulation, in the shape of a Mixed Integer Second Order Cone model,   that constructs Euclidean hyperspheres to identify anomalous observations. Building on this, we develop a dual model that enables the application of the kernel trick, thus allowing for the detection of outliers within complex, non-linear data structures. An extensive computational study demonstrates the effectiveness of our exact method, showing clear advantages over existing heuristic techniques in terms of accuracy and robustness.
\end{abstract}



\section{Introduction}

Support Vector Data Description (SVDD) models were introduced for anomaly detection and one-class classification, as a methodology to recognize patterns presumed as regular, and identifying them from any other object violating this regularity condition. This tool is particularly interesting for datasets where identifying outliers and recognizing data boundaries is required. In contrast to other types of classification tools, the labels of the training observations are not provided, and the goal is to distinguish regular data from anomalous instances by means of geometric boundaries. One of the advantages of SVDD is that it allows the construction of flexible description boundaries without
the need to make assumptions regarding data distribution. The main idea behind (linear) SVDD is to construct a hypersphere in the feature space (its center and its radius) that minimize a trade-off between the radius and a distance from the points \textit{outside} the hypersphere to its boundary. SVDD was introduced by ~\cite{tax2004support} as a tool for one-class classification based on solving a convex optimization problem. The strong duality of the problem was further analyzed by \cite{Wang2011360} implying an equivalent dual reformulation of the problem  that allows for the use of kernels to construct more sophisticated classification boundaries. 

Several improvements have been performed to SVDD since its introduction in order to generate accurate boundaries that are able to capture the information of complex datasets. This is the case of \textit{Density weighted SVDD}~\citep{Cha2014} where the training instances are weighted in the classical SVDD to incorporate the information to the model about the neighbor observations, or the two-phases approach proposed in \citep{liu2013svdd} where these weights are designed to capture the uncertainty in the feature values of the training dataset. \cite{arashloo2022} generalizes the standard Euclidean norm SVDD to a general $\ell_p$-norm penalty function for the errors accounting irregular points (not in the norm defining the hypersphere but in the aggregation of these errors). Furthermore, the analogies with Support Vector Machines (SVM), introduced by Cortes and Vapnik \citep{cortes1995support} are evident, and some of the extensions developed for SVMs have been adapted to SVDD. This is the case of semisupervised SVDD~\citep{mygdalis2018semi} where some of the instances in the dataset are labeled while others are not labeled, Multimodal SVDD~\citep{sohrab2021multimodal} where there are different groups of instances in the dataset with different feature spaces, manifold regularization SVDD for noise filtering~\citep{wu2023manifold}, oversampling approaches for SVDD to address imbalance issues in between-class and within-class imbalances~\citep{tao2022svdd}, as well as different loss functions to aggregate the errors induced by the outlier observations~\citep{zheng2023robust}.

Nevertheless, little attention has been given in the literature to outlier detection methods for multimodal data. That is, the dataset contains instances from different distributions that are to be detected, as well as their outliers. Although the SVDD and their extensions have been proven to be useful to determine the hidden distribution of the given data in case the data comes from a single spatial distribution, these methods fail to detect different distributions. The simultaneous  identification of patterns and anomalous observations is  crucial to understand the distribution of the dataset in multiple fields. For instance, in medical diagnosis is usual to have multimodal distribution of imaging biomarkers (as MRI scans grouped by disease subtypes), and where  outliers might indicate early-stage diseases or rare pathologies that do not conform to known profiles. In Finances, when analyzing the usage of credit card data, multimodal spending patterns are usually present, and  outliers could suggest fraud or unusual behavior. In air transportation or public transit systems, multimodal patterns in movement or load are common, and outliers can flag system inefficiencies, delays, or unauthorized paths.

A few works have analyzed the problem of extending SVDD tools to multiple distributions. The natural extension requires solving a non convex optimization problem, difficulting the analysis and its resolutions. \cite{turkoz2020generalized} provide a multi-class SVDD method by assuming that the assignments of the training data to the different classes is known (a supervised approach). \cite{gornitz2017support} and \cite{xiao2009multi} propose a new methodology to combine SVDD and clustering approaches to derive a machine learning tool to detect outliers observations in this type of datasets. However, all the previously proposed approaches are heuristic since it consists of separately detect distributions (clustering) and outliers (SVDD). 

In this paper, we analyze some   mathematical optimization insights of this problem that we call the MultiSphere Support Vector Data Description (MSVDD, for short) problem. We propose a primal mixed integer nonlinear optimization model for the problem that can be reformulated as a Mixed Integer Second Order Cone Optimization (MISOCO) problem. Then, we derive a dual reformulation that allows for the use of the kernel trick. We derive some mathematical properties of the problem and analyze its ability to detect outlier observations in multimodal instances. 

The use of mathematical optimization tools, particularly discrete optimization models, in machine learning is well established. This synergy has proven successful in deriving interpretable classification rules, as evidenced by applications in support vector machine models~\citep{baldomero2020ejor,baldomero2021robust,blanco2022mathematical,maldonado2014feature1,maldonado2014feature2}, optimal classification trees~\citep{bertsimas2017optimal,blanco2022robust}, and clustering of complex data~\citep{benati2017clustering}, among many others. Here, we take a step forward in this synergy by proposing a novel optimization-based tool for one-class classification, which is particularly well suited for identifying structure in multimodal datasets.

The geometrical problem solved in the (classical) SVDD is closely related with the minimum enclosing ball problem~\citep{sylvester1857question} that consists of finding the hypersphere with the smallest radius enclosing a set of given  points in a Euclidean space. This problem has multiple applications in computer vision~\citep{larsson2016parallel} and machine learning~\citep{wang2007bayes}, and several efficient solution approaches have been proposed~\citep[see e.g.][]{pan2006efficient}. Furthermore, multiple extensions for the problem have been proposed, as the use of generalized shapes instead of hyperspheres~\citep{blanco2021covering} or the minimum $k$-enclosing ball problem that requires only $k$ points to be covered by the hypersphere~\citep{cavaleiro2022branch}. Among all the versions of this problem, the so-called minimum enclosing ball problem with soft constraints~\citep[see e.g.][]{hu2012scaling} coincides with the SVDD. Thus, an alternative interpretation of the SVDD is possible in case the observations represent certain set of customers, and the center of the hypersphere represents a service for those users. Assuming that the agent locating the service pays for the installation of the facility based on its coverage radius, and that for those users outside this radius the  agent pays for a service to transport the users to the coverage area, the problem results in the facility location version of the SVDD. Then, SVDD has also implications in the Facility Location community, which is one of the most active fields in Operations Research~\citep[see the recent book][]{LS2019}. Specifically, one of the natural extensions in this type of applications is to \textit{locate} more than one centers to give service to the users, or equivalently, determine a given number of hyperspheres to cover the demand of the customers, which is 
closely related to what we propose in this extension of the SVDD. The extensions from single-facility to multi-facility problems are not only useful in Facility Location but also have significant implications for Machine Learning, as demonstrated in the context of Fitting Linear Models~\citep{blanco2021multisource} or in multiclass classification~\citep{blanco2020optimal}, where different hidden trends in the data can be captured by simultaneously clustering and fitting/classifying a set of instances.

The main contributions of this paper are:
\begin{itemize}
    \item We provide a \textit{primal} mathematical optimization model to compute a given number of hyperspheres, which define a decision rule for identifying outlier observations in a dataset.
    \item We derive an equivalent \textit{dual} formulation that enables the application of the \textit{kernel trick}, allowing the construction of \textit{kernelized} hyperspheres to detect outliers in complex datasets.
    \item We run an extensive battery of computational experiments to validate our proposal, using both synthetic and real-world instances. We first evaluate the performance of our approach by varying the model's parameters. Then, we compare it with the heuristic location-allocation method proposed by \cite{gornitz2017support}.
\end{itemize}

The rest of the paper is organized as follows. In Section \ref{sec:prelim}, we introduce the problem and set the notation that will be use through this paper. Section \ref{sec:msvdd} is devoted to detail the mathematical optimization formulation for the primal MSVDD, analyzing some of its properties. Then, we derive a dual reformulation that allows the application of the kernel trick. In Section \ref{sec:ce}, we report the results of our computational experiments. Finally, in Section \ref{sec:conc}, we draw some conclusions on the paper and give some possible further research on the topic.

\section{Preliminaries}\label{sec:prelim}

This section introduces 
the problem under analysis and the notation used in the rest of the paper.

We are given a finite training dataset $\X = \{\x_1, \ldots, \x_n\} \subset \mathbb{R}^d$ in a $d$-dimensional feature space. We denote by $N=\{1, \ldots, n\}$ the index set for these observations. 

The SVDD problem, introduced by \cite{tax2004support}, consists of constructing a hypersphere, $\S$, in $\R^d$ such that, ideally, the \textit{regular} observations are contained within $\S$, whereas the outlier observations lie outside $\S$. Note that an hypersphere is fully characterized by its center, $\c\in \R^d$, and its radius $r\geq 0$ as:
$$
\S(\c,r) = \{\x \in \R^d: \|\x-\c\|_2 \leq r\},
$$
where $\|\cdot\|_2$ is the Euclidean norm in $\R^d$.

Given the training dataset, $\X$, and a hypersphere $\S(\c,r)$, the goodness measure that is considered in SVDD is:
$$
r^2 +  \sum_{i\in N} C \max\{0, \|\x_i - \c\|_2^2-r^2\}
$$
where $C\geq 0$ 
is a unit distance penalization parameter for the outlier observations. 
 In this expression, $r^2$ can be interpreted as a fixed cost associated with the hypersphere, reflecting its size and that is independent of the number of observations included in the hypersphere. The second term represents the cumulative hinge-loss penalties incurred by 
the observations with respect to the hypersphere $\S(\c,r)$. Indeed, if an observation 
$\mathbf{x}_i$ lies within the hypersphere, its contribution to this term is zero. However, if $\mathbf{x}_i$ is located outside the hypersphere, the penalty is proportional to the squared Euclidean distance from the point to the hypersphere's boundary.

This objective aims to be minimized by balancing two components: the radius of the hypersphere $\S$ and the penalization for points lying outside it. Consequently, the parameter $C$ serves as a trade-off factor between minimizing the radius of $\S$ and controlling the error associated with outlier observations.

With the above comments, the SVDD problem can be formulated as the following mathematical optimization problem:
\begin{alignat}{4}
    \min \;\;& R + C \sum_{i\in N} \xi_i&&\\
    \mbox{s.t. } & \|\x_i - \c\|_2^2 \leq R + \xi_i,\quad &&\forall i \in N,\\
    & R \geq 0,&&\\
    & \c \in \R^d,&&\\
    & \xi_i \geq 0, &&\forall i \in N,
\end{alignat}
where $R$ represents the squared radius of the hypersphere ($R=r^2$), and $\c$ its center. The variables $\xi_1, \ldots, \xi_n$ provide a measure of the error terms for the outlier observations. As mentioned in \citep{Wang2011360}, the parameter $C$ has also an impact in the number of outlier observations in the training dataset. Specifically, for a given value of $C$, the number of observations outside the hypersphere will be always greater or equal than $1/C$. This is because, otherwise, the objective value can be reduced by decreasing the radius, even if this leads to larger errors.

Since the SVDD formulation above is weakly convex, multiple optimal solutions may exist. Thus, several efficient algorithms have been proposed for solving it for large datasets~\citep[see e.g.][]{jiang2019fast,peng2012efficient}.

Once the problem is solved,  with optimal solutions $R^*$ (radius) and $\c^*$ (center), the decision rule ${\rm D}: \R^d \rightarrow \{\text{Regular}, \text{Outlier}\}$, to detect outliers is the following:
\[
{\rm D}(\x) = \begin{cases}
    \text{Regular,} & \mbox{if $\|\x-\c^*\|_2^2 \leq R^*$},\\
    \text{Outlier,} & \mbox{otherwise}.\\
\end{cases}
\]

In this paper, we analyze the multisphere version of this problem, which can be stated as follows. Given an integer value $p\geq 1$, 
let $P:=\{1, \ldots, p\}$, with $p<n$, the index set for the hyperspheres. The goal of the Multisphere Support Vector Data Description (MSVDD) problem is to construct $p$ hyperspheres, $\S_1, \ldots, \S_p$, in $\R^d$ such that the \textit{regular} observations are contained within at least one of the hyperspheres, whereas the outliers observations are in $\R^d \backslash \bigcup_{j\in P} \S_j$.

Hence, 
given the training dataset, $\X$, and a set of hyperspheres $\S_1, \ldots, \S_p$, the goodness measure we use is defined as the sum of the squared radii of 
each hypersphere, plus the errors corresponding to the outliers observations.

Unlike the single-hypersphere case, where error is measured directly with respect to that 
single sphere, in MSVDD, the error induced by an outlier observation must be calculated with respect to the hypersphere which it is assigned to.

\section{A mathematical optimization model for MSVDD}\label{sec:msvdd}

In this section, we propose a mathematical optimization model for the  MSVDD and a Mixed Integer Second Order Cone reformulation for it. We analyze some of its properties and construct a dual reformulation that allows the application of the kernel trick to the method.

In order to formulate the MSVDD problem as a suitable mathematical optimization problem, we consider the following set of binary variables:
\[
z_{ij} = \begin{cases}
    1, & \mbox{if $\x_i$ is allocated to hypersphere $\S_j$,}\\
    0, & \mbox{otherwise,}
\end{cases}  \quad \forall i \in N, j \in P.
\]

These variables will allow us to model the hypersphere with respect to which an outlier's error is measured.

With these variables, alongside the continuous variables representing the centers ($\c_j\in \R^d$, for $j\in P$), the squared radii ($R_j\geq 0$, for $j\in P$), the error measures ($\xi_i\geq 0$, for $i\in N$), and the parameter $C$, we propose the following optimization model for the MSVDD:

\begin{alignat}{4}
\rm (F_{MSVDD})\quad\min \;\; & \sum_{j\in P} R_j + \sum_{i \in N}  C\xi_i \label{msvdd:obj}\\
\mbox{s.t. } & \|\x_i - \c_j\|_2^2  \leq R_j + \xi_i + \Delta_i(1-z_{ij}),\quad && \forall i \in N, j \in P, \label{msvdd:2}\\
& \sum_{j\in P} z_{ij}=1, &&\forall i \in N,\label{msvdd:1}\\
& \xi_i \geq 0, &&\forall i \in N, \label{msvdd:35}\\
& \c_j \in \mathbb{R}^d, &&\forall j \in P, \label{ctr:7_0}\\
& R_j \geq 0, &&\forall j \in P,\label{ctr:8_0}\\
& z_{ij} \in \{0,1\}, &&\forall i \in N, j \in P, \label{ctr:9_0}
\end{alignat}
\noindent
where $\Delta_i$ is a big enough constant, we can assume that  $\Delta_i := \max_{i'\in N}\|\x_i - \x_{i'}\|_2^2$, for any $i \in N$. The objective function \eqref{msvdd:obj} accounts for the sum of the square radii of the hyperspheres and the penalty term for each outlier defined as $C$ times the distance to its assigned hypersphere's boundaries. 

Constraints \eqref{msvdd:2} are a linearization, using McCormick envelopes~\citep{mccormick1976computability}, of:
 $$ \|\x_i - \c_j\|_2^2 z_{ij} \leq R_j + \xi_i,\quad \forall i \in N, j \in P.$$
Then, these constraints adequately define the error $\xi_i$, for $i\in N$: in case $\x_i$ is allocated to hypersphere $\S_j$ ($z_{ij}=1$), i.e., the constraint is activated, indicating  that $\xi_i = \max\{0, \|\x_i-\c_j\|_2^2 - R_j\}$. 
 Constraint \eqref{msvdd:1} 
assures that each observation is allocated to exactly one hypersphere.

Note that  constraints \eqref{msvdd:2} can be rewritten as a Second Order Cone constraints~\citep{alizadeh2003second}, and then, the problem above results in a MISOCO problem.

\begin{remark}
    Given an optimal solution of  \rm ($F_{MSVDD}$), $(\bar \r, \bar \c, \bar \xxi, \bar \z)$, we construct the outlier detection rule ${\rm D}: \R^d \rightarrow \{\text{Regular}, \text{Outlier}\}$ as follows:
\[
{\rm D}(\x) = \begin{cases}
    \text{\rm Regular,} & \mbox{if $\exists j \in P: \|\x-\bar \c_j\|_2^2 \leq \bar R_j$},\\
    \text{\rm Outlier,} & \mbox{otherwise.}\\
\end{cases}
\]

\end{remark}
Some properties of the solutions to (F$_\text{MSVDD}$) are analyzed below.

\begin{propo}
    Let $(\bar \r, \bar \c, \bar \xxi, \bar \z)$ a feasible solution to the  \rm ($F_{MSVDD}$) with objective value $f^*$. Then, for all permutations $\sigma$ of $P$, the tuple  $(\tilde \r, \tilde \c, \bar \xxi, \tilde \z)$, where
    \begin{alignat}{4}
    \tilde R_j  \: &=& \:\bar R_{\sigma(j)}, \quad&\forall j \in P,\\
    \tilde \c_{j}\: & =& \:\bar c_{\sigma(j)},\quad  &\forall j \in P,\\
    \tilde z_{ij}\: &=& \:\bar z_{i \sigma(j)}, \quad &\forall i \in N, j \in P,
    \end{alignat}
    has also an objective value of $f^*$.
\end{propo}
\begin{proof}
    The result is straightforward.
\end{proof}

Based on the above result, one has that multiple (at least $p!$) optimal solutions of the problem appear. In order to accelerate the solution procedure of the problem, one can incorporate the following symmetry breaking constraints to the problem to avoid this multiplicity of solutions:
$$
R_{j} \leq R_{j+1}, \forall j \in P\backslash\{p\}.
$$

\begin{propo}\label{prop:allocation}
Let $(\bar \r, \bar{\mathbf{c}}, \bar \xxi, \bar{\mathbf{z}})$ be an optimal solution of \text{(F$_\text{MSVDD}$)}. Then,
$$
\sum_{i \in N} \bar z_{ij} \geq 1, \quad \forall j \in P.
$$
\end{propo}
\begin{proof}
Let $(\bar \r, \bar{\mathbf{c}}, \bar \xxi, \bar{\mathbf{z}})$ be an optimal solution of \textnormal{(F$_\text{MSVDD}$)} such that $\sum_{i \in N} \bar z_{ij_0} = 0$ for some $j_0 \in P$. In this case, we have $\bar R_{j_0} = 0$, since otherwise the objective value could be reduced by setting $R_{j_0} = 0$, as no points are assigned to this sphere.

In what follows, we prove that it is possible to obtain a solution with a better objective value where the (degenerated) hypersphere $\S_{j_0}$ with $\bar R_{j_0}=0$ has a point assigned to it.   

Let us consider the solution $(\hat \r, \hat \c,\hat \xxi, \hat\z )$, with:
    \begin{align*}
         \hat \c_j &= \bar \c_j, \forall j \in P\backslash\{j_{0}\}\\
        \hat R_j &= \bar R_j, \forall j \in P\backslash\{j_{0}\}.
    \end{align*}
    \noindent and $f^*$ the objective value in the optimal solution.
    
   We analyze first the case in which there exists $i_0 \in N$ with $\bar\xi_{i_{0}}>0$.  To construct the hypersphere $\S_{j_0}$, we set $\hat\c_{j_0} = \x_{i_0}$ and $\hat R_{j_0}=0$. Then, we define the errors $\hat \xxi_i$ as follows:
    $$
    \hat \xi_i = \min_{j\in P} \max \{0, \|\x_i - \hat \c_j\|-\hat R_j\}, \forall i \in N.
    $$
    Note that $\hat \xxi_{i_0}=0$. Then, we set
    $$
    \hat z_{ij} = \begin{cases}
        \bar z_{ij}, & \mbox{if $i\neq i_0$,}\\
        1, & \mbox{if $i=i_0$, $j=j_0$,}\\
        0, & \mbox{if $i=i_0$, $j\neq j_0$.}
    \end{cases}
    $$
Note that the objective value of this solutions is:
\begin{eqnarray*}
\sum_{j \in P} \hat R_j + \sum_{i\in N} C \hat \xi_i &\stackrel{\hat\xi_{i_0}=0, \hat{R}_{j_0}=0}{=}& \Big(\sum_{j \in P\backslash\{j_{0}\}} \hat R_j + \sum_{i\in N \backslash \{i_{0}\}} C \hat \xi_i\Big) \\
&=&f^* - C \bar \xi_{i_{0}} \\
&<& f^*.
\end{eqnarray*}
Thus, contradicting the optimality of the solution.

Secondly, in case $\bar\xi_i=0$ for all $i \in N$, since $p<n$, one can construct an equivalent solution by setting $\hat\c_{j_0}=\x_i$ for an arbitrary $i\in N$, with the same optimal value, but enforcing the allocation of at least one point to each hypersphere.
\end{proof}

\begin{propo}\label{propo:C}
Let $(\bar \r, \bar \c, \bar \xxi, \bar \z)$ be an optimal solution of a relaxed version of {\rm ($F_{MSVDD}$)} without constraints \eqref{ctr:8_0}, and $j \in P$.
If $\displaystyle \sum_{i\in N}  \bar{z}_{ij}\geq \frac{1}{C}$, then $\bar{R}_j\geq 0$. 

Furthermore, if $\displaystyle  \sum_{i\in N} \bar{z}_{ij} = \frac{1}{C}$, then there exists an optimal solution with $\bar R_j=0$.
\end{propo}
\begin{proof}
We assume that $(\bar \r, \bar \c, \bar \xxi, \bar \z)$ is an optimal solution of {\rm ($F_{MSVDD}$)}. Let us assume that there exists $j_0\in P$, such that $\sum_{i\in N}  \bar z_{ij_0}\geq \frac{1}{C}$ with $\bar R_{j_0}<0$.

Let us construct an alternative feasible solution, $(\hat{\r}, \hat{\c}, \hat{\xxi}, \hat{\z})$, as follows:
$$
\hat{\c}:=\bar \c,  \hat{\z}:=\bar\z
$$
and:
\begin{eqnarray}\label{sol}
\hat{\xi}_i = \begin{cases}
    \bar \xi_i + \bar R_{j_0}, & \mbox{if } \bar{z}_{ij_0}=1,\\
    \bar \xi_i , & \mbox{otherwise,}
\end{cases} \forall i \in N, && \hspace*{-0.45cm}\hat{R}_j = \begin{cases}
    \bar R_j, & \mbox{if } j\neq j_0,\\
    0, & \mbox{otherwise}, 
\end{cases}\forall j \in P.
\end{eqnarray}

Observe that  $\xi_i + \bar R_{j_0}\ge 0$ when $\bar{z}_{ij_0}=1$, since it is the right hand side of  \eqref{msvdd:2} for $j=j_0$.

Let us prove that $(\hat{\r}, \hat{\c}, \hat{\xxi}, \hat{\z})$  satisfies \eqref{msvdd:2}. Since the alternative solution that we constructed only differs from the previous one in the indices $i\in N$ with $\hat z_{ij_{0}} = \bar z_{ij_{0}}=1$, it is suffice to prove that it holds for these indices in $N$:
\begin{eqnarray*}
\|\x_i - \hat{\c}_{j_0}\|_2^2 =\|\x_i - \bar \c_{j_0}\|_2^2   \leq (\bar R_{j_0} + \bar\xi_i) \stackrel{\eqref{sol}}{=} \hat{R}_{j_0} + \hat{\xi}_i  \stackrel{\hat z_{ij_{0}}=1}{=} \hat{R}_{j_0} + \hat{\xi}_i+ \Delta_i(1-\hat{z}_{ij_0}).
\end{eqnarray*}
Thus  $(\hat{\r}, \hat{\c}, \hat{\xxi}, \hat{\z})$ is feasible.

Let $\bar f := \sum_{j\in P} \bar {R}_j + C\sum_{i \in N} \bar {\xi}_i$ be optimal value of the problem. The objective value of the new feasible solution is:
\begin{eqnarray*}
\hat{f} := \sum_{j\in P} \hat{R}_j + C\sum_{i \in N} \hat{\xi}_i&=&
 \sum_{j\in P \atop j\neq j_0} \bar{R}_j + C\sum_{i \in N} \bar{\xi}_i+ C \sum_{i \in N\atop \bar{z}_{ij_0}=1} \bar R_{j_0} = \bar f - \bar R_{j_{0}} + C\sum_{i \in N\atop \bar{z}_{ij_0}=1} \bar R_{j_0}.
\end{eqnarray*}
Since we are asuming that $\sum_{i\in N} \bar z_{i{j_0}} \geq \frac{1}{C}$ and $\bar R_{j_{0}}<0$, we have that $\displaystyle C\sum_{i \in N\atop \bar{z}_{ij_0}=1} \bar R_{j_0} = C \bar R_{j_{0}}\sum_{i \in N} \bar{z}_{ij_0} \leq \bar R_{j_{0}}$. Therefore, $\hat{f} \leq \bar f$, i.e., $(\hat{\r}, \hat{\c}, \hat{\xxi}, \hat{\z})$ is also optimal with  $\hat{R}_{j_{0}}\geq 0$.

Note that if $\sum_{i\in N} \bar z_{i{j_0}} = \frac{1}{C}$, then $\hat{f} = \bar f$, and in that case, another optimal solution with $\hat{R}_{j_{0}} = 0$ exists.
\end{proof}

\subsection{Dual Reformulation}

It is well known that for complex datasets, the geometric clustering obtained from a solution to the primal formulation may fail to capture the true structure of the training data. To address this limitation, particularly in high-dimensional or nonlinearly separable cases, dual formulations are often employed. A key advantage of the dual approach is that it enables the use of kernel functions, allowing the method to operate implicitly in a (possibly much) higher-dimensional feature space without explicitly computing the transformation. Specifically, by assuming the existence of a mapping $\Phi: \R^d \rightarrow \R^{d'}$, we can perform outlier detection or clustering in the transformed space $\mathbb{R}^{d'}$, where linear separators correspond to more complex boundaries in the original input space. The  kernel trick allows us to express inner products $\langle \Phi(x), \Phi(x') \rangle$ through a kernel function, which is particularly advantageous for constructing more flexible, data-adaptive decision boundaries than Euclidean hyperspheres in the original space.

By Proposition \ref{propo:C}, we conclude that a way to ensure non-negative radii appear in the solution is to allocate at least $\frac{1}{C}$ points to each hypersphere. Therefore, from now on, we will impose that the constraint $\sum_{i\in N}  {z}_{ij} \geq \frac{1}{C}$ holds for each $j \in P$, which implies that the non-negativity of the radii need not be explicitly enforced through additional constraints. 

Hence, with all the comments above, in  what follows we analyze the following (primal) version of the MSVDD problem in the transformed feature space using mapping $\Phi$:
\begin{alignat}{4}
\hspace*{-0.5cm} {\rm (F^{\Phi}_{MSVDD})}\,\min \;\; & \sum_{j\in P} R_j +  \sum_{i \in N} C\xi_i \nonumber\\
\mbox{s.t. } & \eqref{msvdd:1}, \eqref{msvdd:35},\eqref{ctr:9_0} \nonumber\\
& 
\|\Phi(\x_i) - \c_j\|_2^2  \leq R_j + \xi_i + \Delta_i(1-z_{ij}),
\quad 
&& \forall i \in N, j \in P, \label{msvdd:p2}\\
& {{\sum_{i\in N} C z_{ij}  \geq 1,}} && \forall j \in P,\label{msvdd:3}\\
& \c_j \in \mathbb{R}^{d'}, &&\forall j \in P.\label{msvdd:36}
\end{alignat}

For solving the above MSVDD problem in the image space via $\Phi$, 
the next result provides a reformulation of the problem to apply the so-called kernel trick, that does not require the explicit expression of the transformation $\Phi$.

Let us assume that we are given $\K: \R^d \times \R^d \rightarrow \R$, a kernel function verifying Mercers' condition~\citep{mercer} that represent $\K(\x, \mathbf{y}) = \Phi(\x)^T \Phi(\mathbf{y})$ for all $\x, \mathbf{y} \in \R^d$. 
Let us denote by $\mathcal{K}$-MSVDD the dual problem of MSVDD using the kernel function $\mathcal{K}$.

\begin{theo} Formulations 
{($F^{\Phi}_{MSVDD}$)}   and  {($F_{MSVDD}^{\K}$)} (below) are equivalent in the sense that both formulations have the same optimal value and an optimal solution from one formulation can be built from an optimal solution of
the other, where 
\begin{alignat}{4}
{\rm (F^{\K}_{MSVDD})} \: \min &\ \sum_{j\in P} R_j + \sum_{i \in N} C\xi_i \nonumber\\
\mbox{s.t. } & \eqref{msvdd:1}, \eqref{msvdd:35},\eqref{ctr:9_0} \nonumber \\
& \K(\x_i,\x_i) - 2 \sum_{k\in N} \alpha_{kj} \K(\x_i,\x_k) + \sum_{k\in N} \sum_{l\in N} && \alpha_{kj}\alpha_{lj} \K(\x_k,\x_l) \:\;
\nonumber\\ &
\leq R_j + \xi_i + \Delta_i(1-z_{ij}), &&\forall i \in N, j \in P, \label{dual_ctr}\\
&0 \leq \alpha_{ij} \leq Cz_{ij}, &&\forall i \in N, j \in P,\label{alphaC}\\
&\sum_{i\in N} \alpha_{ij}= 1, &&\forall j \in P, \label{sumalpha}
\end{alignat}
where $\Delta_i$ is a large enough constant, that 
 can be set as:
 $$
 \displaystyle\Delta_i := \K(\x_i,\x_i)+2 \sum_{k\in N} \pi_{ik}\K(\x_i,\x_k) + \sum_{k\in N} \sum_{l\in N} (C-\pi_{kl})^2 \K(\x_k,\x_l),
 $$
 for all $i \in N$, and where:
 $$
 \pi_{ik} = \begin{cases}
     C,  & \text{\rm if } \K(\x_i,\x_k) < 0, \\
     0,  &  \rm{otherwise.}
 \end{cases}
 $$
\end{theo}
\begin{proof}

Let $(\r^*, \c^*, \xxi^*,  \z^*)$ be an optimal solution of ($F^{\Phi}_{MSVDD}$), and let ($F^{\Phi}_{MSVDD}(\z^*)$) the formulation ($F^{\Phi}_{MSVDD}$) where the $\z$-variables have been fixed to the $\z^*$-values, and constraints \eqref{msvdd:1},
\eqref{ctr:9_0} and \eqref{msvdd:3} have been removed since they only depend on the these variables.

Thus, ($F^{\Phi}_{MSVDD}(\z^*)$) can be reformulated as
\begin{alignat}{4}
\min \;\; & \sum_{j\in P} R_j +  C \sum_{i \in N}\xi_i \nonumber\\
\mbox{s.t. } &  \eqref{msvdd:35}, \eqref{msvdd:36}, \nonumber \\
&
{\| \Phi(x_i) - \c_j\|_2^2 z^*_{ij} \leq (R_j + \xi_i)z^*_{ij},}\quad && \forall i \in N, j \in P. \label{msvdd*:22}
\end{alignat}

This problem is a convex non-linear problem and its Lagrangian function is given by:
\begin{align}
 \mathcal{L}_{\z^*}(\r,\c,\xxi;\aalpha,\bbeta)&=\sum_{j\in P} R_j + C \sum_{i \in N} \xi_i \label{lagrangian} \\
&+{\sum_{i\in N}\sum_{j\in P} \alpha_{ij}z^*_{ij} (\|\Phi(\x_i) - \c_j\|^2 - R_j - \xi_i)}  - \sum_{i\in N} \beta_{i}\xi_{i},\nonumber
\end{align}
where {$\alpha_{ij} \geq 0$}, for $i\in N$, $j\in P$, and $\beta_i \geq 0$ for $i\in N$ are Lagrange multipliers associated with constraints \eqref{msvdd*:22} and 
\eqref{msvdd:35}, respectively.

Thus, the Lagrange dual is
 \begin{align*}
    \displaystyle\max_{\bf{\aalpha,\bbeta}\geq 0}\left(\inf_{\c,\r,\xxi} \mathcal{L}_{\z^*}(\r,\c,\xxi;\aalpha,\bbeta)\right)
 \end{align*}
The optimality conditions, setting partial derivatives with respect the $\r$, $\c$, and $\xxi$-variables to zero, result in the following relations:
\begin{align}
 & \frac{\partial \mathcal{L}_{\z^*}}{\partial R_j}(\r,\c,\xxi;\aalpha,\bbeta) = 0 \Rightarrow 1 - \sum_{i\in N} \alpha_{ij} z^*_{ij} = 0\Rightarrow \sum_{i\in N} \alpha_{ij} z^*_{ij}= 1, \label{der:rj_0}\\
  & 
  \frac{\partial \mathcal{L}_{\z^*}}{\partial c_j}(\r,\c,\xxi;\aalpha,\bbeta)\hspace*{-0.1cm} = \hspace*{-0.1cm} 0 \hspace*{-0.1cm} \Rightarrow \hspace*{-0.25cm} {\sum_{i\in N} \alpha_{ij}z^*_{ij}(c_j - \Phi(x_i)) = 0} 
 \hspace*{-0.1cm} \Rightarrow \hspace*{-0.1cm} \c_j \hspace*{-0.1cm} = \hspace*{-0.25cm} \displaystyle\frac{\displaystyle\sum_{i\in N} \alpha_{ij} z^*_{ij} \Phi(x_i)}{\displaystyle\sum_{i\in N} \alpha_{ij}z^*_{ij}}, \label{centers_0}\\
 & \frac{\partial \mathcal{L}_{\z^*}}{\partial \xi_i}(\r,\c,\xxi;\aalpha,\bbeta) = 0 \Rightarrow \hspace*{-0.1cm} {{C} - \hspace*{-0.15cm} \sum_{j\in P} \alpha_{ij} z^*_{ij} - \beta_i = 0} \Rightarrow \hspace*{-0.25cm}{{\sum_{j\in P} \alpha_{ij} z^*_{ij}={C} - \beta_i,} }\label{der_chi}
 \end{align}
 Note that by Proposition \ref{prop:allocation}, the system defined by the above equations is well-defined.
Since $\alpha_{ij} \geq 0$, for $i\in N$, $j\in P$, and $\beta_i \geq 0$ for $i\in N$, then by \eqref{der_chi}, we obtain that  $\sum_{j\in P} \alpha_{ij} z^*_{ij} \leq C $. 
{Moreover, since $\sum_{j \in P} z^*_{ij}=1$ with $z^*_{ij}\in\{0,1\}$, and $\alpha_{ij}$ always appears multiplied by $z^*_{ij}$, then we can assume that\begin{eqnarray}
0 \leq  \alpha_{ij} \leq {C} z^*_{ij}. \label{alpha_1}\end{eqnarray}}
Since, by Proposition \ref{prop:allocation}, $\sum_{i\in N} z^*_{ij}\geq 1$, $\forall j\in P$,   \eqref{der:rj_0} and \eqref{centers_0}  can be rewritten as:
\begin{alignat}{4}
&\sum_{i\in N} \alpha_{ij}=1,\quad&&\forall j \in P, \label{oc:3}\\
&\c_j  = \sum_{i\in N} \alpha_{ij} \Phi(\x_i), \quad&&\forall j \in P,
\label{c_phi}
\end{alignat}

Then, a dual formulation to ($F^{\Phi}_{MSVDD}(z^*)$) is given by:
\begin{alignat*}{4}
\max \;\; &\sum_{j \in P} \Big( \sum_{i \in N} \alpha_{ij} \Phi(\x_i)^T\Phi(\x_i) - \sum_{k\in N}\sum_{\ell \in N} \alpha_{kj}\alpha_{\ell j}\Phi(\x_k)^T\Phi(\x_{\ell}) \Big)\\
\mbox{s.t. } & \eqref{alpha_1}, \eqref{oc:3}.
\end{alignat*}
Let $\aalpha^*$ be an optimal solution of this problem. We prove that $(\r^*,\xxi^*, \aalpha^*, \z^*)$, where $\c^*$ is defined by \eqref{c_phi}, is a feasible solution for   $(F^{\K}_{MSVDD})$. Let us prove that \eqref{dual_ctr} is satisfied. Since 
$\c^*_j  = \sum_{i\in N} \alpha^*_{ij} \Phi(\x_i)$, from 
\eqref{msvdd:p2}  we obtain 
\begin{align}\label{eq:17}
\nonumber R_j + \xi_i + \Delta_i(1-z_{ij})&\geq \|\Phi(\x_i) - \c_j\|^2 = \Phi(\x_i)^T\Phi(\x_i) - 2\sum_{k\in N} \alpha_{kj}\Phi(\x_i)^T\Phi(\x_k) \\&+ \sum_{k\in N} \alpha_{kj}\Phi(\x_k)\sum_{l\in N} \alpha_{lj}\Phi(\x_l).
\end{align}

Given a kernel function $\K: \R^d \times \R^d \rightarrow \R$ verifying  the Mercer's conditions~\citep{mercer} with $\K(\x, \x')= \Phi(\x)^t \Phi(\x')$, the expression of \eqref{eq:17} reads:
\begin{align*}
\|\Phi(\x_i) - \c_j\|^2 = \K(x_i,x_i) - 2 \sum_{k\in N} \alpha_{kj} \K(x_i,x_k) + \sum_{k\in N} \sum_{l\in N} \alpha_{kj}\alpha_{lj} \K(x_k,x_l),
\end{align*}
for all $i\in N$, and \eqref{dual_ctr} is satisfied. Moreover, $(\r^*,\xxi^*, \aalpha^*, \z^*)$ has the same objective value as
the optimal value of $(F^{\Phi}_{MSVDD})$. 

Conversely, considering $(\r^*,\xxi^*, \aalpha^*, \z^*)$ an optimal solution of $(F^{\K}_{MSVDD})$,  we prove that
 $(\r^*, \c^*, \xxi^*,  \z^*)$, where $\c^*$ is defined by \eqref{c_phi}, is a feasible solution of ($F^{\Phi}_{MSVDD}$) with the same objective value. Constraints \eqref{msvdd:p2} hold and in addition, \eqref{msvdd:3} are verified since $\sum_{i\in N} C z^*_{ij} \geq \sum_{i\in N} \alpha^*_{ij} = 1$, for all $j\in P$. 
\end{proof}

\begin{remark}
    Observe that, given an optimal solution of ($F^{\K}_{MSVDD}$), $(\bar \r,\bar \xxi, \bar\aalpha, \bar \z)$ the decision rule  is based on checking whether an observation $\x$ verifies $\|\Phi(\x)-\bar \c_j\|^2 \leq \bar R_j$, where $\bar \c_j  = \sum_{i\in N} \bar \alpha_{ij} \Phi(\x_i)$, for some $j\in P$. Since the $\Phi$ function may not be known, this check can be done by means of the kernel function $\mathcal{K}$ by identifying the squared Euclidean norm above with the following expression:
\begin{align*}
\|\Phi(\x)-\bar c_j\|_2^2 &= \|\Phi(\x)- \sum_{i\in N} \bar \alpha_{ij} \Phi(\x_i)\|_2^2 \\
&= \K(\x,\x) - 2 \sum_{i\in N} \bar \alpha_{ij} \K(\x,\x_i) + \sum_{k\in N} \sum_{i\in N} \bar \alpha_{kj}\bar \alpha_{ij} \K(\x_k,\x_i)
\end{align*}
Thus, the decision rule induced by the dual problem is:
$$
{\rm D}(\x) = \begin{cases}
    \text{\rm Regular,} & \text{\rm if } \exists j \in P: \K(\x,\x) - 2 \displaystyle\sum_{i\in N} \bar \alpha_{ij} \K(\x,\x_i)\\
    & \mbox{\qquad $+ \displaystyle\sum_{k\in N}\displaystyle\sum_{i\in N}\bar\alpha_{kj}\bar \alpha_{ij} \K(\x_k,\x_i) \leq\bar R_j$},\\
    \text{\rm Outlier,} & \text{\rm otherwise.}\\
\end{cases}
$$
The expression above can be checked using exclusively the values of the dual variables, $\bar \aalpha$, the radii, $\bar \r$, and the training sample through the kernel function.
\end{remark}

The dual formulation $(F^{\K}_{MSVDD})$ extends the classical dual weighed SVDD model~\citep{tax2004support}, as stated in the following result:

\begin{propo}
The formulation $(F^{\K}_{MSVDD})$ for the case of just one sphere
 is equivalent to the classical dual formulation of SVDD.
\end{propo}
\begin{proof}
The formulation $(F^{\K}_{MSVDD})$ for the case of just one sphere
 is given by
\begin{alignat}{4}
{\rm ({F_1^{\K}}_{MSVDD})} \:  \min &\  R + \sum_{i \in N}C \xi_i \nonumber\\
\: \mbox{s.t. }  & \eqref{msvdd:35}, \nonumber \\
& \K(\x_i,\x_i) - 2 \sum_{k\in N} \alpha_{k} \K(\x_i,\x_k) + \sum_{k\in N} \sum_{l\in N}  \alpha_{k}\alpha_{l} \K(&& \x_k, \x_l) \nonumber \\
& \leq R + \xi_i, &&\forall i \in N, \label{unajdual_ctr}\\
&\sum_{i\in N} \alpha_{i}= 1. && \label{unajsumalpha} \\
&0 \leq \alpha_{i} \leq C, &&\forall i \in N,\label{unajalphaC}
\end{alignat}
where the $j$-index have been removed to simplify the notation for a single  sphere.
Moreover,  the classical dual formulation of SVDD, (F$^{\K}_{\text{DSVDD}}$), is given by,
\begin{alignat*}{4}
 {\rm (F^{\K}_{DSVDD})} \quad \max &\   -\sum_{k \in N} \sum_{l\in N} \hat{\alpha}_{k} \hat{\alpha}_{j} \K(\x_k,\x_l) + \sum_{k\in N} \hat{\alpha}_{k} \K(\x_k,\x_k) \nonumber\\
\mbox{s.t. } 
&\sum_{k\in N} \hat{\alpha}_{k}= 1, && \\
&0 \leq \hat{\alpha}_{k} \leq C, &&\forall k \in N.
\end{alignat*}
Indeed, multiplying each constraint of \eqref{unajdual_ctr} times $\alpha_i$ and taking the sum of all them, we obtain that
$$\sum_{i \in N} \alpha_i \K(\x_i,\x_i) - 2 \sum_{i \in N} \sum_{k\in N} \alpha_i \alpha_{k} \K(\x_i,\x_k) + \sum_{k\in N} \sum_{l\in N} \alpha_{k}\alpha_{l} \K(\x_k,\x_l) 
\leq R + \sum_{i \in N} \alpha_i \xi_i.$$
Or equivalently,
$$\sum_{i \in N} \alpha_i \K(\x_i,\x_i) -  \sum_{i \in N} \sum_{k\in N} \alpha_i \alpha_{k} \K(\x_i,\x_k)  \leq R + \sum_{i \in N} \alpha_i \xi_i \leq R + \sum_{i \in N} C \xi_i.$$
Thus, any feasible solution $(\r,\xi,\alpha)$  of  ${\rm ({F_1^{\K}}_{MSVDD})}$ verifies that  $(\xxi,\aalpha)$ is feasible for $(F^{\K}_{DSVDD})$ and the objective value for the former is greater or equal than the objective value for  the latter. Furthemore, from an optimal solution $(\hat{\xxi}^*,\hat{\aalpha}^*)$ of $(F^{\K}_{DSVDD})$, there exists an optimal primal solution of the classical SVDD, $(\hat{\r}^*,\hat{\c}^*,\hat{\xxi}^*)$, with the same objective value, i.e., 
$$\sum_{i \in N} \hat{\alpha}^*_i \K(\x_i,\x_i) -  \sum_{i \in N} \sum_{k\in N} \hat{\alpha}^*_i \hat{\alpha}^*_{k} \K(\x_i,\x_k) = \hat{R}^* + \sum_{i \in N}  C \hat{\xi}^*_i,$$
and $(\hat{\r}^*,\hat{\xxi}^*, \hat{\aalpha}^*_i)$ is feasible for ${\rm ({F_1^{\K}}_{MSVDD})}$ because 
\eqref{unajdual_ctr} is equivalent to the primal constraing $\|\Phi(\x_i)-\hat{\c}^*\|\le\hat{R}^*_i+\hat{\xi}_i$ by considering $\hat{\c}^*=\sum_{i\in N} \hat{\alpha}^*_i \Phi(\x_i)$.
\end{proof}

%


\section{Computational Experiments}\label{sec:ce}

{

In this section, we validate the proposed MSVDD approaches through a series of computational experiments on both synthetic and real-world datasets.

The proposed primal and dual formulations were coded in Python 3.11 and solved using Gurobi 12.0 \citep[see][]{Gurobi2025}. The computational experiments were performed on a Windows 10 platform with an Intel Core i9-12900K processor (3.2 GHz) and 32 GB of RAM. The default settings of Gurobi were used, with a time limit of 3600 seconds imposed for solving each training instance in the experiments.

\subsection{Synthetic Datasets}

We randomly generated a set of instances following the scheme described by \cite{gornitz2017support}. Each dataset consists of planar coordinates sampled from two isotropic Gaussian distributions with standard deviations $0.5$ and $0.6$, respectively. These datasets are referred to as \textit{regular} data. To introduce noise, we added anomalous observations by fixing a noise level in $\{5\%, 10\%, 15\%, 20\%\}$. For each noise level, the corresponding proportion of points was uniformly sampled and placed close to, but outside of the regular clusters. In Figure \ref{f:data_test}, we show one of the instances with $15\%$ of anomalous observations. For this purpose, each instance consisted of $100$ training data points, $66$ validation data points, and $166$ testing data points, including the percentage of anomalous observations.

\begin{figure}[h!]
    \centering
    \includegraphics[width=0.48\linewidth]{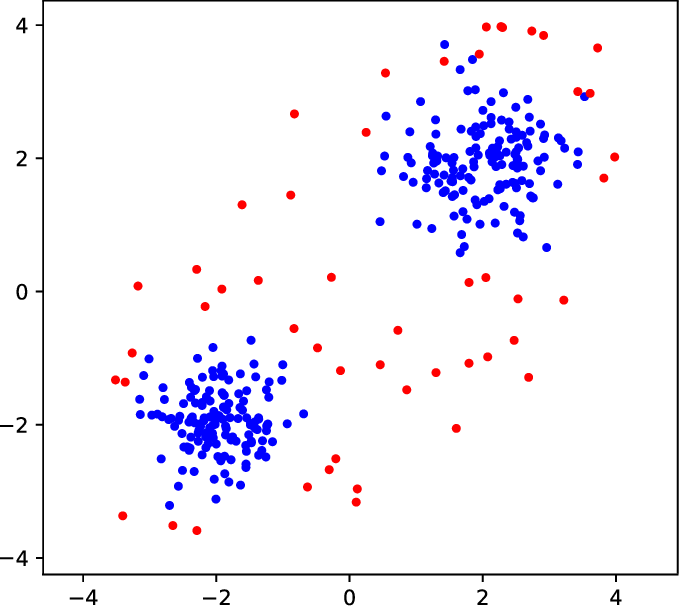}
    \caption{Example of synthetic dataset with 15\% of anomalous data. Blue points are labeled as \textit{regular} data and red points as \textit{anomalous}.
    }
    \label{f:data_test}
\end{figure}

First, we analyze the computational efficiency (in terms of CPU time) of our approaches. Figure~\ref{f:perf_prof} presents the performance profiles for both MSVDD and $\K$-MSVDD as we vary the parameters $p$ and $C$. In all four plots, the $x$-axis represents the required CPU time, while the $y$-axis indicates the percentage of instances optimally solved within that time. We observe that for MSVDD, instances with larger values of $p$ are more challenging. This is expected, as increasing $p$ directly leads to a more complex combinatorial structure and a higher number of binary variables in the model. Conversely, we find that as the value of $C$ decreases, the primal problems become more computationally demanding. As we will discuss later, a smaller $C$ implies a greater number of observations identified as outliers, which in turn increases the combinatorial complexity of the problem. 

For the dual model, we extended the analysis of the previous instances to different values of $\sigma$. The results indicate that the performance trend remains consistent for values of the regularization parameter $C$ up to approximately $0.20$. However, beyond this value, increasing $C$ makes the instances increasingly challenging.

\begin{figure}[h!]
		\begin{center}
			\subfigure[MSVDD varying $p$]
			{
				\includegraphics[scale=0.12]{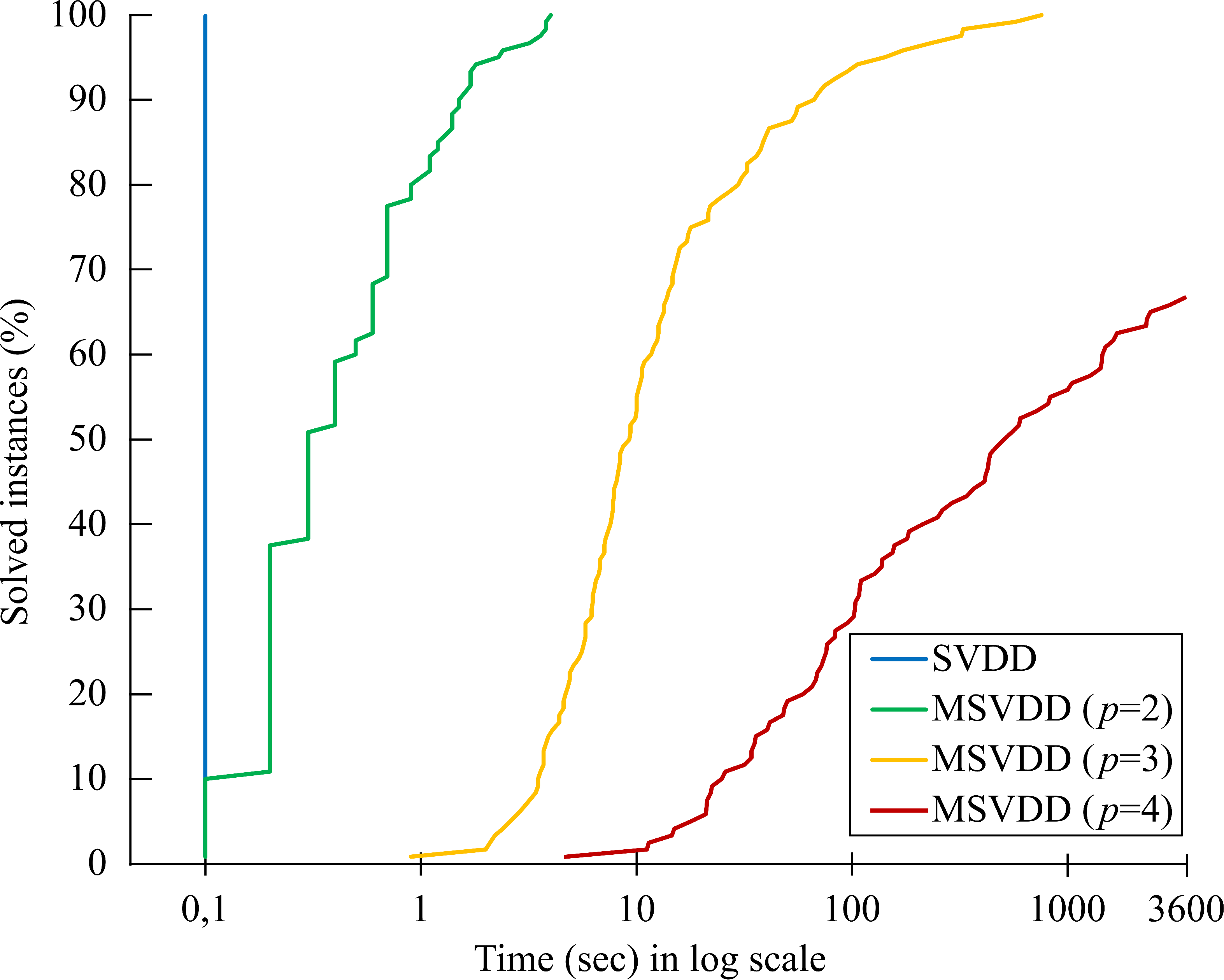}
				\label{f:perf_prof_primal_p}
				\vspace*{0.2cm}}
			\subfigure[MSVDD varying $C$]
			{
				\includegraphics[scale=0.12]{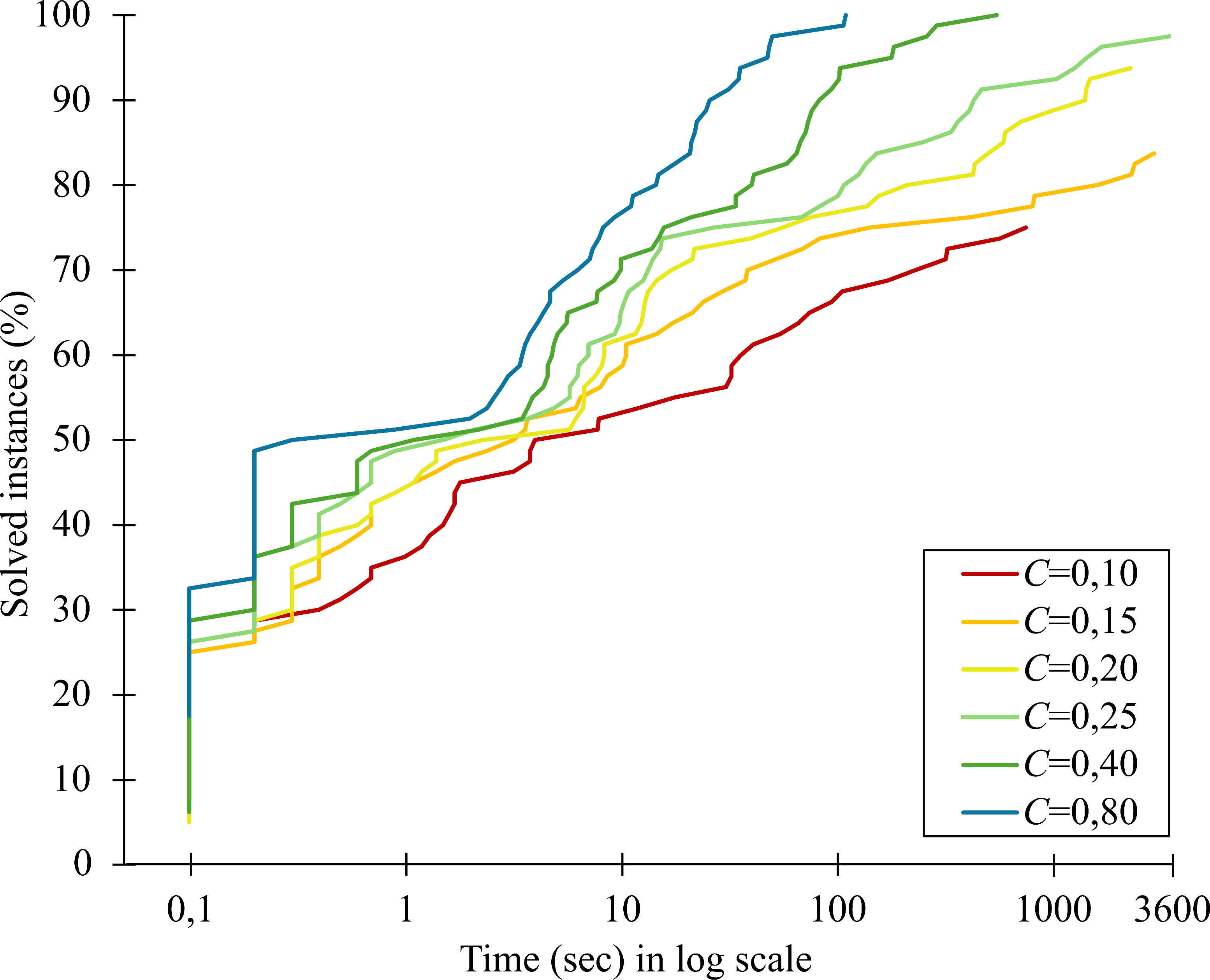}
				\label{f:perf_prof_primal_C}
				\vspace*{0.1cm}}
            \subfigure[$\K$-MSVDD varying $p$]
			{
				\includegraphics[scale=0.12]{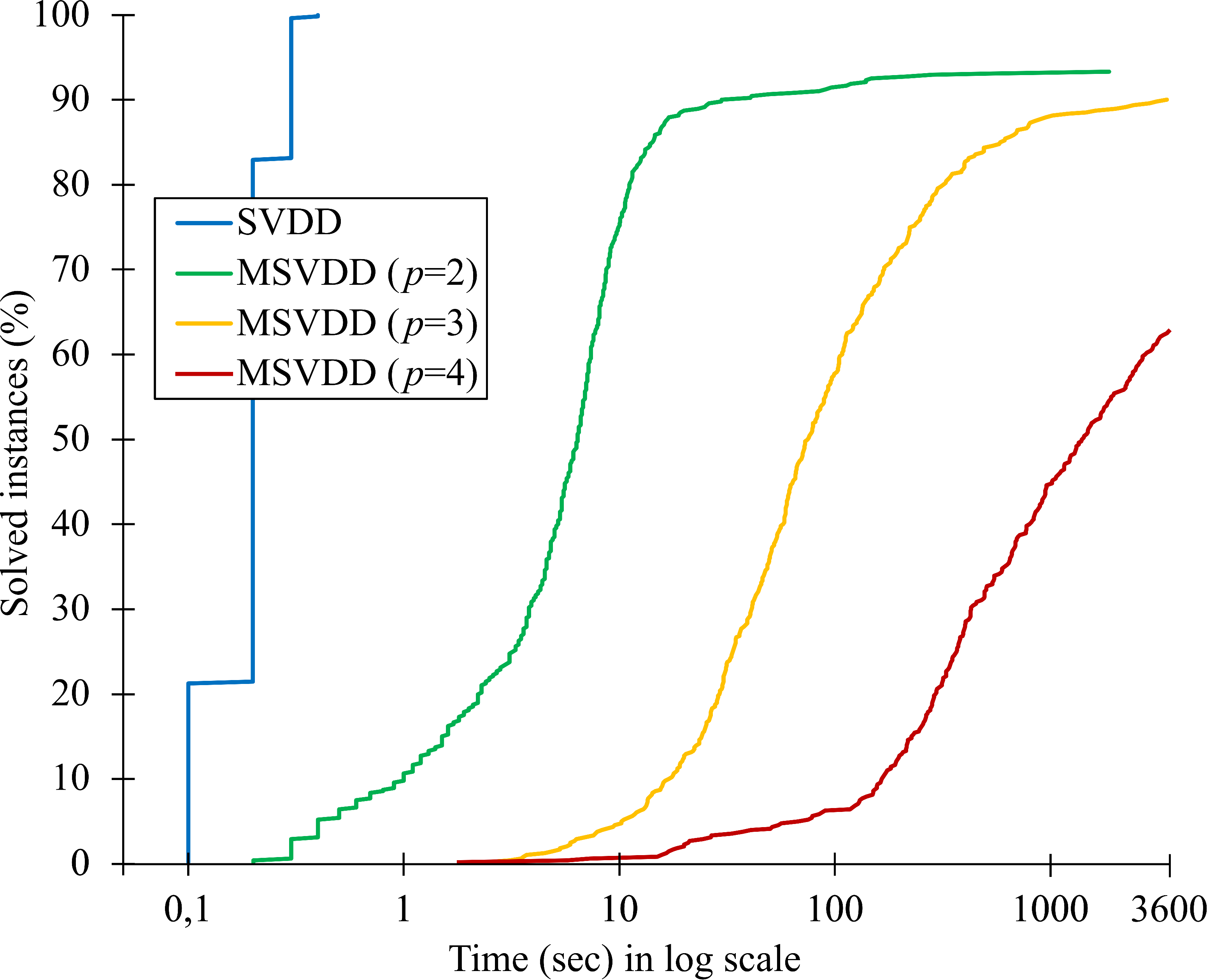}
				\label{f:perf_prof_dual_p}
				\vspace*{0.2cm}}
			\subfigure[$\K$-MSVDD varying $C$]
			{
				\includegraphics[scale=0.12]{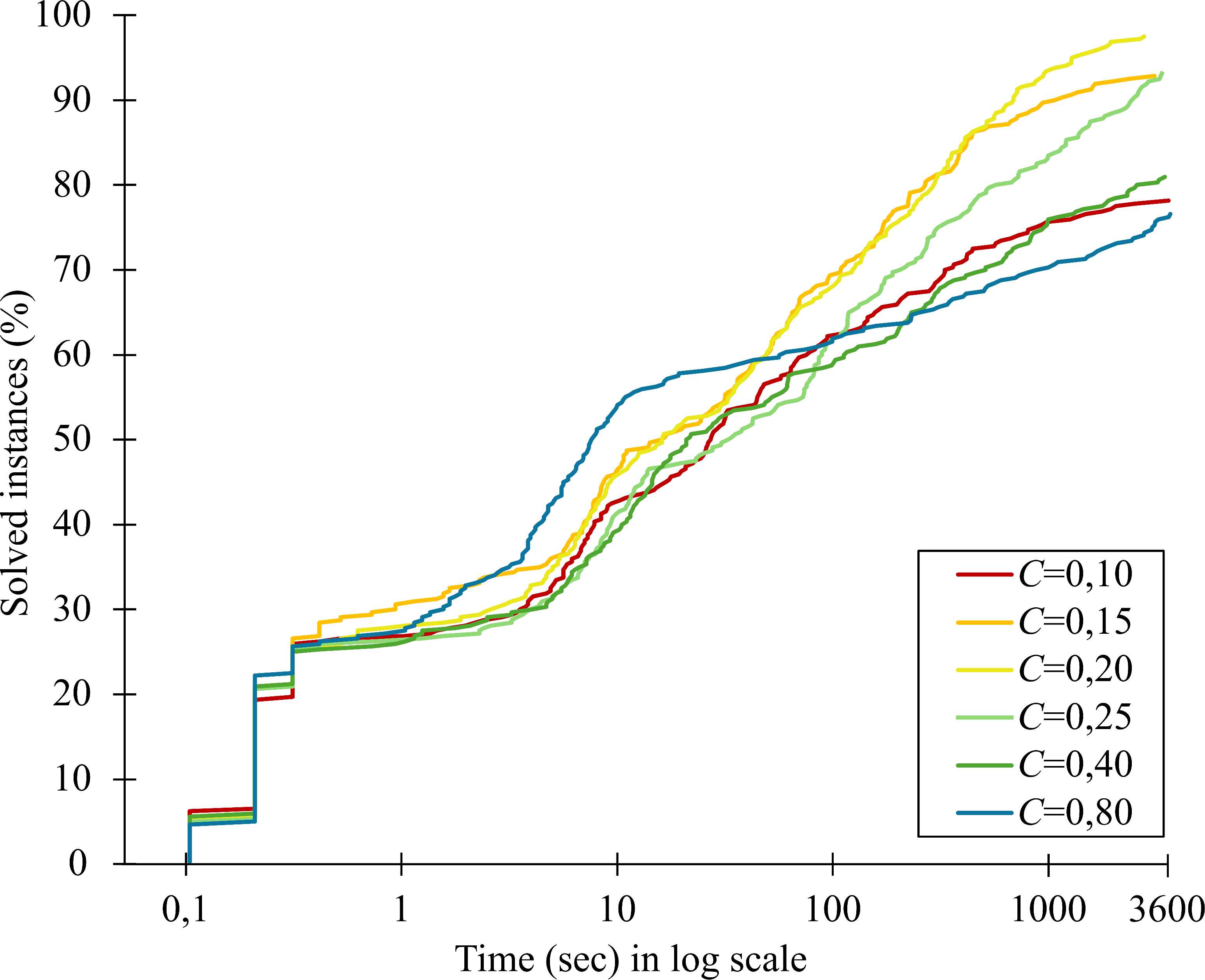}
				\label{f:perf_prof_dual_C}
				\vspace*{0.1cm}}

			\caption{Performance profile for MSVDD and $\K$-MSVDD varying $p$ and regularization parameter $C$.
            }  
    \label{f:perf_prof}
		\end{center}
\end{figure}

In what follows, we evaluate the performance of our approach as a machine learning tool for identifying anomalous data within a sample, comparing it to existing methods for multimodal data, specifically, the heuristic procedure \texttt{ClusterSVDD} proposed by \cite{gornitz2017support}. In that  approach, the procedure begins with a random partitioning of the data into $p$ clusters. Subsequently, a single-sphere SVDD model is trained on each cluster, with a predefined proportion of outliers relative to the total number of observations. The procedure then alternates between updating the cluster assignments and re-estimating the SVDD models in the new resulting clusters, until the assignments no longer change from the previous iteration. In \texttt{ClusterSVDD} } the regularization parameter applied to each cluster is set as $C = 1/\nu N_k$, where $N_k$ is the number of training observations for the hypersphere $k$, with $k\in \{1, ...,p\}$, and $\nu$ denotes the assumed fraction of outliers within each cluster. Since the number of data points assigned to each cluster $k$ is known in advance, the number of outliers per sphere can be directly controlled.

In contrast, our exact optimization model simultaneously determines the hyperspheres, the assignment of points to each hypersphere, and the identification of trainning outliers. As a result, the number of points in each hypersphere is not predefined, so we cannot directly specify the number of anomalous points per hypersphere beforehand. Consequently, there is not a correspondence between the regularization parameter $C$ of our MSVDD and the $\nu$ paramenter in \texttt{ClusterSVDD}. Therefore, we separately determine their optimal values through cross-validation. However, to conceptually relate our regularization parameter $C$ to \texttt{ClusterSVDD}'s $\nu$, particularly for the synthetic datasets with components of roughly equal numbers of data points, we can consider the approximate relationship $C = p / \nu N$.

For instance, considering the examples shown in Figure~\ref{f:gor_nu_0.05}, with $N = 100$ training points, with  5\% of outliers, we must set  $\nu = 0.05$ in \texttt{ClusterSVDD}. In this scenario, using the expression $C=p/\nu N$, the value of $C$ for our model would be $0.4$ (see Figure~\ref{f:ms_C_0.4}), which differs from $0.2$ (see Figure~\ref{f:ms_C_0.2}) as implied by the relationship $C = 1/\nu N$ (for $p=1$). Notably, with $C = 0.4$, our method achieves a comparable objective value and identifies the same number of outliers as the heuristic method, although not the same outlier observations.

\begin{figure}[h!]
	\begin{center}
        \subfigure[$\nu=0.05$, $O.V.=5.3562$, $\#outliers=6$.]
        {
            \includegraphics[scale=0.48]{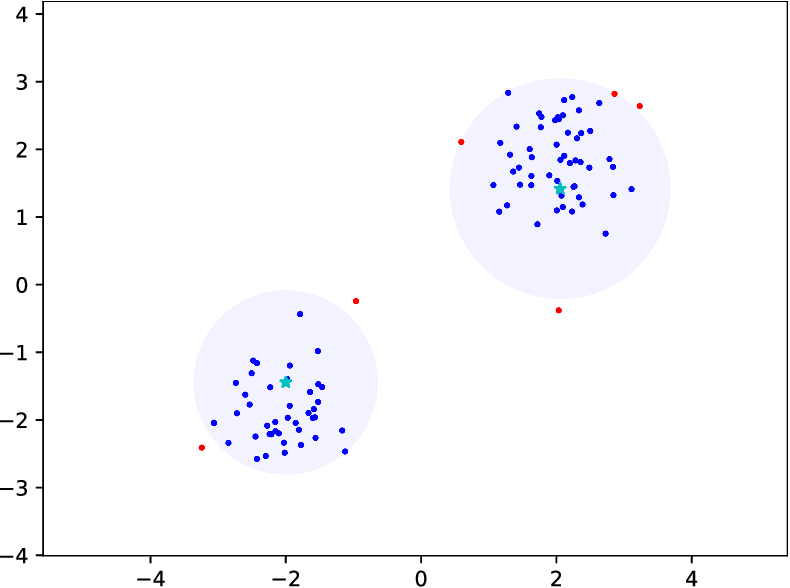}
            \label{f:gor_nu_0.05}
        } \\
        \vspace*{0.1cm}
        \subfigure[$C=0.2$, $O.V.=4.3334$, $\#outliers=9$.]
        {
            \includegraphics[scale=0.48]{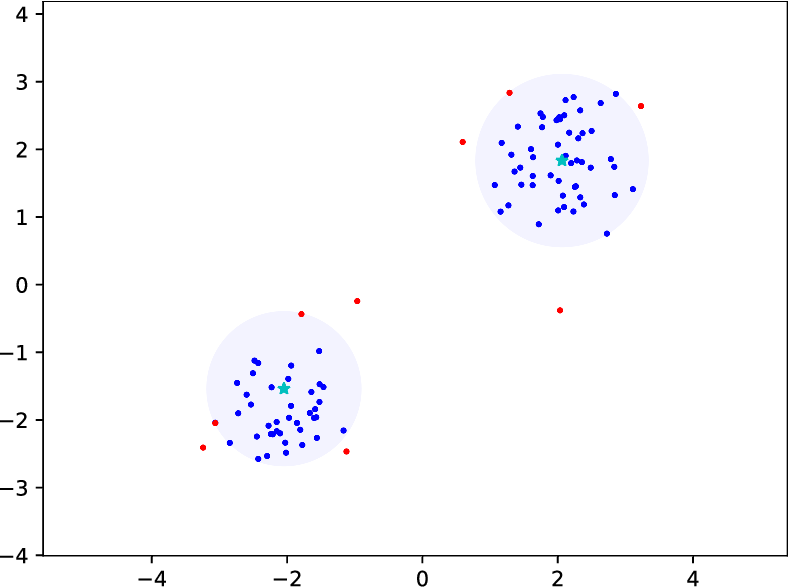}
            \label{f:ms_C_0.2}
        }
        \hfill
        \subfigure[$C=0.4$, $O.V.=5.2748$, $\#outliers=6$.]
        {
            \includegraphics[scale=0.48]{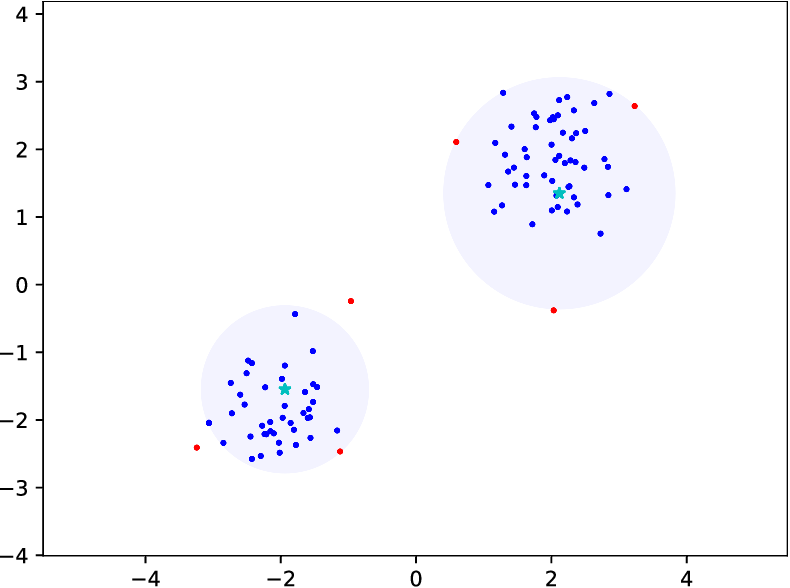}
            \label{f:ms_C_0.4}
        }
        \caption{Solutions obtained for $p=2$ with \texttt{ClusterSVDD} (top) and MSVDD (bottom). Dataset with  $N=100$ and 5\% of anomalous data. Blue points are classified as \textit{regular} data and red points as \textit{outliers}. 
        }
        \label{f:ms_C}
	\end{center}
\end{figure}

For each instance, we then run \texttt{ClusterSVDD} (using their code publicly available at GitHub\footnote{\url{https://github.com/nicococo/ClusterSvdd}}) for $\nu$-values in $
\{0.025, 0.05, 0.075, 0.1, 0.15, 0.2\}$ 
and our MSVDD method for the $C$-parameters in
$\{0.1, 0.15, 0.2, 0.25, 0.4, 0.8\}$.

Additionally, we run the kernelized dual version of our model for the RBF kernel with parameter $\sigma^2 \in \{0.05, 0.1, 0.25, 0.5\}$. Then, since the observations of our synthetic instances are already labeled as regular/anomalous, we compute the AUC-ROC (Area Under the Receiver Operating Characteristic Curve). The AUC-ROC evaluates the ability of a scoring function to distinguish between regular and outliers observations. Specifically, it is defined as the probability that a randomly chosen regular observation has a higher score than a randomly chosen anomalous observation. Formally, for a set of predicted scores, the AUC is computed by comparing the true positive rate and false positive rate across all possible thresholds. A value of AUC equal to $1$ indicates perfect discrimination, while a value of $0.5$ corresponds to random guessing.

Tables \ref{t:loose_prim} and \ref{t:loose_dual} summarize the AUC-ROC scores obtained for instances with 5\%, 10\%, 15\%, and 20\% anomalous data, presenting results for the primal and dual versions of the models, respectively. The first column indicates the model used: \texttt{ClusterSVDD} and MSVDD. The second column displays the percentage of anomalous data added to the dataset. The subsequent four columns report the mean and standard deviation of the AUC-ROC scores obtained from five instances for each evaluated $p$ value. The corresponding optimal $\nu$ (for \texttt{ClusterSVDD}) or $C$ (for MSVDD) value, determined via cross-validation, is indicated in brackets alongside the scores.

As an example of parameter cross-validation, Figure \ref{f:prim_aucroc} display the mean and standard deviation of the AUC-ROC scores obtained from five instances, which were generated with 15\% anomalous data points. After parameter cross-validation, the optimal regularization parameters for the proposed MSVDD method were determined to be $C=0.15$ and $p=3$, in the case of primal model (see Figure \ref{f:ms_prim_aucroc}).  On the other hand,  cross-validation for \texttt{ClusterSVDD} indicated that the optimal parameter combination for training that model was $\nu=0.15$ with $p=3$  (see Figure \ref{f:gor_prim_aucroc}). 

Regarding the accuracy of the primal version of both methodologies (see Table \ref{t:loose_prim}), when $p=1$ (corresponding to the classical single-sphere SVDD), both models exhibit comparable performance, although 
MSVDD typically achieves slightly higher accuracy and lower variance. As the number of hyperspheres increases, the performance advantage of 
MSVDD becomes more pronounced. Notably, for $p=3$ and $p=4$, the proposed method consistently attains near-perfect accuracy, even under high noise levels, while maintaining very low variability across runs. In contrast, \texttt{clusterSVDD} struggles to maintain its performance in these more complex configurations, with a clear drop in accuracy and a noticeable increase in variability. These results demonstrate the robustness and effectiveness of the exact optimization formulation in capturing the structure of the data and isolating anomalous observations, particularly when multiple hyperspheres are used to model heterogeneous distributions.

\begin{table}[htp]
    \makebox[\textwidth][c]{%
        {\footnotesize
        \renewcommand{\arraystretch}{1.2} 
         \adjustbox{scale=0.7}{\begin{tabular}{
        l
        c
        S[table-format=1.4] @{ / } S[table-format=1.9, table-align-text-post=false]
        S[table-format=1.4] @{ / } S[table-format=1.9, table-align-text-post=false]
        S[table-format=1.4] @{ / } S[table-format=1.9, table-align-text-post=false]
        S[table-format=1.4] @{ / } S[table-format=1.9, table-align-text-post=false]
    }
    \hline
    Model & {\% anom.} & \multicolumn{2}{c}{$p=1$ (SVDD)} & \multicolumn{2}{c}{$p=2$} & \multicolumn{2}{c}{$p=3$} & \multicolumn{2}{c}{$p=4$} \\
    \hline
        \texttt{ClusterSVDD}  &  5   &  0.3293 & 0.1992\ (0.075) &  0.8053 & 0.3749\ (0.2)  &  0.9831 & 0.0098\ (0.2)  &  0.8119 & 0.2188\ (0.15) \\      
        ${\text{MSVDD}}$    &  5   &  0.3624 & 0.1721\ (0.8)   &  0.9920 & 0.0074\ (0.1)  &  0.9925 & 0.0042\ (0.2)  &  \bfseries 0.9943 & \bfseries 0.0055\ (0.2)  \\ \hline    
        \texttt{ClusterSVDD}  &  10  &  0.4828 & 0.0948\ (0.075) &  0.9371 & 0.0378\ (0.1)  &  0.8771 & 0.1011\ (0.2)  &  0.7488 & 0.1485\ (0.2)  \\     
        ${\text{MSVDD}}$    &  10  &  0.4830 & 0.0816\ (0.3)   &  0.9801 & 0.0290\ (0.1)  &  \bfseries 0.9982 & \bfseries 0.0027\ (0.1)  &  0.9974 & 0.0022\ (0.1)  \\ \hline     
        \texttt{ClusterSVDD}  &  15  &  0.5354 & 0.0638\ (0.05)  &  0.8547 & 0.1452\ (0.15) &  0.8836 & 0.0904\ (0.15) &  0.8179 & 0.0868\ (0.1)  \\     
        ${\text{MSVDD}}$    &  15  &  0.5354 & 0.0611\ (0.15)  &  0.9673 & 0.0170\ (0.1)  &  \bfseries 0.9860 & \bfseries 0.0016\ (0.15)  &  0.9836 & 0.0202\ (0.15) \\ \hline      
        \texttt{ClusterSVDD}  &  20  &  0.4152 & 0.0583\ (0.2)   &  0.6967 & 0.1947\ (0.1)  &  0.7880 & 0.1594\ (0.2)  &  0.7094 & 0.1611\ (0.1)  \\     
        ${\text{MSVDD}}$    &  20  &  0.4255 & 0.0534\ (0.8)   &  0.9070 & 0.0736\ (0.1)  &  \bfseries 0.9615 & \bfseries 0.0329\ (0.1)  &  0.9531 & 0.0297\ (0.15) \\  
        \hline
        \end{tabular}}%
        }
    }
    \caption{Mean and standard deviation of AUC-ROC values obtained after cross-validation for instances with synthetic data using primal models. The corresponding optimal $\nu$ or $C$ value is indicated in brackets. 
    }
    \label{t:loose_prim}
\end{table}

The parameter cross-validation results for the dual versions (see Table \ref{t:loose_dual}) demonstrate that ${\K}$-MSVDD  consistently outperforms $\K$-\texttt{ClusterSVDD} (the dual/kernelized versions of MSVDD and \texttt{ClusterSVDD} for the RBF kernel indicated above), specially as the number of hyperspheres increases. This model not only achieves significantly higher AUC-ROC scores, indicating a better anomaly discrimination capability, but also exhibits greater stability (reflected in lower standard deviations) in its results. The combination of MSVDD with the kernel trick allows its decision boundary to adapt more flexibly and precisely to the shape of the data, which has a positive impact on the accuracy of detecting outliers in the training sample. However, this increased flexibility may also lead to overfitting on the test sample, as evidenced by the AUC-ROC values reported in the Table \ref{t:loose_dual}. Actually, when comparing the AUC-ROC scores obtained from the optimal solutions of the primal and dual models, by varying the $\sigma^2$ parameter in the dual, while keeping the proportion of anomalous data and the parameters $p$ and $C$ fixed in both, the dual model outperformed the primal in 70.6\% of the cases.

\begin{table}[htp]
    \makebox[\textwidth][c]{%
        {\footnotesize
        \renewcommand{\arraystretch}{1.2} 
         \adjustbox{scale=0.7}{
         \begin{tabular}{
        l
        c
        S[table-format=1.4] @{ / } S[table-format=1.9, table-align-text-post=false]
        S[table-format=1.4] @{ / } S[table-format=1.9, table-align-text-post=false]
        S[table-format=1.4] @{ / } S[table-format=1.9, table-align-text-post=false]
        S[table-format=1.4] @{ / } S[table-format=1.9, table-align-text-post=false]
    }
    \hline
    Model & {\% anom.} & \multicolumn{2}{c}{$p=1$ (SVDD)} & \multicolumn{2}{c}{$p=2$} & \multicolumn{2}{c}{$p=3$} & \multicolumn{2}{c}{$p=4$} \\
    \hline
    $\K$-\texttt{ClusterSVDD}   &  5   &  0.9824 & 0.0135\ (0.2)   &  0.9837 & 0.0111\ (0.15)  &  0.9830 & 0.0119\ (0.025) &  0.9810 & 0.0130\ (0.15)  \\
    ${\K}$-MSVDD     &  5   &  0.9753 & 0.0170\ (0.1)   &  0.9842 & 0.0170\ (0.8)   &  0.9849 & 0.0114\ (0.8)   &  \bfseries 0.9946 & \bfseries 0.0045\ (0.8)   \\ \hline
    $\K$-\texttt{ClusterSVDD}   &  10  &  0.9740 & 0.0292\ (0.1)   &  0.9832 & 0.0176\ (0.2)   &  0.9792 & 0.0190\ (0.1)   &  0.9788 & 0.0221\ (0.075) \\
    ${\K}$-MSVDD     &  10  &  0.9759 & 0.0255\ (0.1)   &  0.9833 & 0.0154\ (0.4)   &  0.9806 & 0.0200\ (0.8)   &  \bfseries 0.9818 & \bfseries 0.0197\ (0.8)   \\ \hline
    $\K$-\texttt{ClusterSVDD}   &  15  &  0.9456 & 0.0293\ (0.2)   &  0.9664 & 0.0128\ (0.025) &  0.9617 & 0.0132\ (0.05)  &  0.9623 & 0.0137\ (0.15)  \\
    ${\K}$-MSVDD     &  15  &  0.9508 & 0.0231\ (0.1)   &  0.9665 & 0.0200\ (0.2)   &  0.9637 & 0.0175\ (0.4)   &  \bfseries 0.9730 & \bfseries 0.0135\ (0.8)   \\ \hline
    $\K$-\texttt{ClusterSVDD}   &  20  &  0.9291 & 0.0179\ (0.025) &  0.9329 & 0.0290\ (0.075) &  0.9278 & 0.0184\ (0.025) &  0.9282 & 0.0178\ (0.1)   \\
    ${\K}$-MSVDD     &  20  &  0.9255 & 0.0196\ (0.1)   &  0.9224 & 0.0209\ (0.8)   &  \bfseries 0.9418 & \bfseries 0.0243\ (0.8)   &  0.9392 & 0.0233\ (0.8)   \\
    \hline
    \end{tabular}
    }%
    }
    }
    \caption{Mean and standard deviation of AUC-ROC values obtained after cross-validation for instances with synthetic data using dual models. The corresponding optimal $\nu$ or $C$ value is indicated in brackets.
    }
    \label{t:loose_dual}
\end{table}



\begin{figure}[h!]
    \centering
    \makebox[\linewidth]{%
        \begin{minipage}{\textwidth + 3cm} 
            \centering 
            \subfigure[clusterSVDD]
            {
                \includegraphics[scale=0.5]{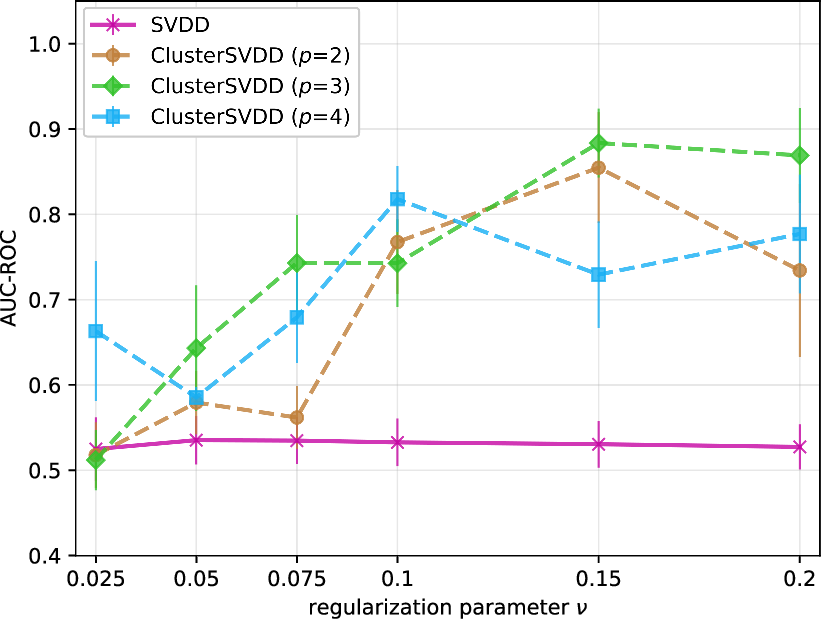}
                \label{f:gor_prim_aucroc}
                \vspace*{0.2cm}
            }
            \subfigure[MSVDD]
            {
                \includegraphics[scale=0.5]{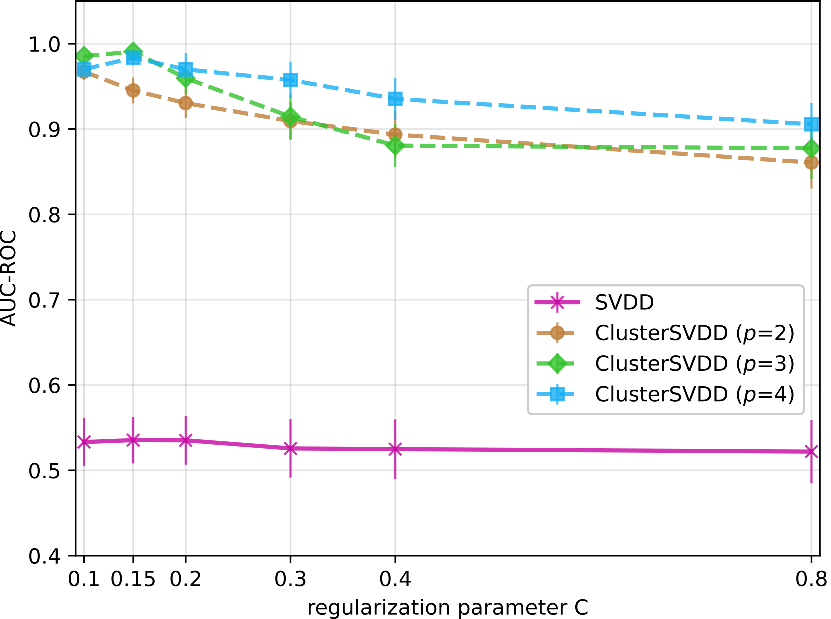}
                \label{f:ms_prim_aucroc}
                \vspace*{0.1cm}
            }
        \end{minipage}%
    }
    \caption{AUC-ROC mean values for \texttt{ClusterSVDD} (left) and ${\text{MSVDD}}$ (left) with $p=1$ (standard SVDD), $p\in \{2, 3, 4\}$, for varying regularization parameter $\nu$ and $C$, respectively. Anomalous fraction of 15\% in the dataset.}
    \label{f:prim_aucroc}
\end{figure}

\subsection{Real Datasets}

To evaluate the proposed model MSVDD on real-world datasets and compare its performance with \texttt{ClusterSVDD}, we selected three real-world multi-class datasets sourced from the \texttt{libSVM}\footnote{\url{https://www.csie.ntu.edu.tw/~cjlin/libsvmtools/datasets/}} library, see Table \ref{tab:datasets}), where all features are scaled to the range $[-1, 1]$.

For each dataset, the selected classes were partitioned into training (30\%), validation (20\%), and test (50\%) sets, all designated as regular data. Subsequently, varying percentages of anomalous data (5\%, 10\%, 15\%, and 20\%) were added, as detailed for each specific case below. To conduct cross-validation over distinct pairs of parameters $p$, $C$ and $\nu$ (see Table \ref{tab:reg_param}), five distinct instances were generated for each dataset by randomly partitioning the regular data and randomly introducing anomalous data in each iteration.

The first dataset employed is the widely recognized Iris database. This dataset comprises 3 classes with a total of 150 data points. The specified percentages of anomalous data were randomly introduced into these instances.

Next, we employed the Ionosphere dataset, which contains two classes and a total of 351 data points. Positive class (225 data points) was designated as regular data, and the specified percentages of anomalous data were randomly extracted from negative class.

Finally, the Segment dataset was employed. This dataset contains 7 classes with a total of 2310 data points. Three of these classes (totaling 990 data points) were labeled as regular data, while the corresponding percentages of anomalous data were extracted from the remaining four classes.

\begin{table}[h!]
\makebox[\textwidth][c]{%
{\footnotesize
\renewcommand{\arraystretch}{1.2} 
 \adjustbox{scale=0.85}{\begin{tabular}{lcccp{6cm}} 
\hline
Dataset & Classes & \#Instances & Features & Description \\
\hline
\textit{Iris} & 3 & 150 (50 per class) & 4 & Measurements of iris flowers from three species (setosa, versicolor, virginica). \\ \hline
\textit{Ionosphere} & 2 & 225 (class 1), 126 (class -1) & 34 & Radar signal returns from the ionosphere classified as either good or bad reflections. \\ \hline
\textit{Segment} & 7 & 2310 (330 per class) & 19 & High-level numeric descriptors of segmented regions from seven outdoor images. \\
\hline
\end{tabular}}
}
}
\caption{Details of the datasets used.}
\label{tab:datasets}
\end{table}

Table \ref{tab:results_real_DB} summarizes cross-validation results. Theses results demonstrate that for every evaluated parameters combination (i.e., pairs of $p$ and $C$ for MSVDD, and $p$ and $\nu$ for \texttt{clusterSVDD}), MSVDD consistently outperforms \texttt{clusterSVDD} across all datasets and anomalous data percentages, achieving higher AUC-ROC scores, particularly with $p\geq 2$ or $p\geq 3$. Furthermore, MSVDD generally exhibits greater stability, as indicated by its lower standard deviations when achieving optimal performance.

\begin{table}[ht]
\centering
{\footnotesize
\renewcommand{\arraystretch}{1.2} 
\begin{tabular}{l l}
\hline
libSVM dataset & Regularization parameters\\
\hline
\textit{Iris} & $C = \{0.07,\ 0.08,\ 0.09,\ 0.1,\ 0.11,\ 0.12,\ 0.13\}$ \\
              & $\nu = \{0.05,\ 0.075,\ 0.1,\ 0.125,\ 0.15,\ 0.175,\ 0.2\}$ \\ \hline
\textit{Ionosphere} & $C = \{0.025,\ 0.05,\ 0.1,\ 0.2,\ 0.3,\ 0.4,\ 0.5\}$ \\
              & $\nu = \{0.05,\ 0.075,\ 0.1,\ 0.125,\ 0.15,\ 0.175,\ 0.2\}$ \\ \hline
\textit{Segment} & $C = \{0.025,\ 0.05,\ 0.1,\ 0.2,\ 0.3,\ 0.4,\ 0.5\}$ \\
              & $\nu = \{0.05,\ 0.075,\ 0.1,\ 0.125,\ 0.15,\ 0.175,\ 0.2\}$ \\
\hline
\end{tabular}
}
\caption{Regularization parameter values used in each real dataset.}
\label{tab:reg_param}
\end{table}

\begin{table}[h!]
\makebox[\textwidth][c]{%
{\footnotesize
\renewcommand{\arraystretch}{1.2} 
\centering
 \adjustbox{scale=0.85}{
\begin{tabular}{c|c|c|c|c|c}
\hline
\textbf{Dataset} & \textbf{Model} & \% Anom. & SVDD & $p=2$ & $p=3$ \\
\hline
\multirow{8}{*}{\rotatebox{90}{Iris}} 
  & \texttt{clusterSVDD} & 5 & 0.5693 / 0.1389 & 0.7880 / 0.1359 & 0.8267 / 0.1646 \\
  & MSVDD       & 5 & 0.5760 / 0.1452 & 0.8787 / 0.1041 & \textbf{0.9440 / 0.0394} \\
  & \texttt{clusterSVDD} & 10 & 0.5937 / 0.0762 & 0.8093 / 0.0915 & 0.8423 / 0.0978 \\
  & MSVDD       & 10 & 0.5970 / 0.0785 & \textbf{0.9133 / 0.0204} & 0.9080 / 0.0455 \\
  & \texttt{clusterSVDD} & 15 & 0.6167 / 0.0559 & 0.8575 / 0.0569 & 0.8698 / 0.0417 \\
  & MSVDD       & 15 & 0.6145 / 0.0820 & 0.8924 / 0.0380 & \textbf{0.9091 / 0.0455} \\
  & \texttt{clusterSVDD} & 20 & 0.5849 / 0.0250 & 0.8404 / 0.1137 & 0.8329 / 0.0786 \\
  & MSVDD       & 20 & 0.5902 / 0.0292 & \textbf{0.8450 / 0.1570} & 0.8363 / 0.1645 \\
\hline
\multirow{8}{*}{\rotatebox{90}{Ionosphere}} 
  & \texttt{clusterSVDD} & 5 & 0.8068 / 0.0821 & 0.8513 / 0.0746 & 0.8578 / 0.1100 \\
  & MSVDD       & 5 & 0.8174 / 0.0608 & 0.8523 / 0.0649 & \textbf{0.8782 / 0.0815} \\
  & \texttt{clusterSVDD} & 10 & 0.7408 / 0.0766 & 0.7855 / 0.0575 & 0.7673 / 0.0871 \\
  & MSVDD       & 10 & 0.7492 / 0.0908 & 0.7966 / 0.0594 & \textbf{0.8454 / 0.0394} \\
  & \texttt{clusterSVDD} & 15 & 0.7731 / 0.0590 & 0.8073 / 0.0382 & 0.7847 / 0.0536 \\
  & MSVDD       & 15 & 0.7749 / 0.0900 & 0.8125 / 0.0985 & \textbf{0.8330 / 0.0714} \\
  & \texttt{clusterSVDD} & 20 & 0.8188 / 0.0485 & 0.8303 / 0.0668 & 0.8599 / 0.0370 \\
  & MSVDD       & 20 & 0.8260 / 0.0524 & 0.8423 / 0.0392 & \textbf{0.8823 / 0.0654} \\
\hline
\multirow{8}{*}{\rotatebox{90}{Segment}} 
  & \texttt{clusterSVDD} & 5 & 0.9971 / 0.0013 & 0.8991 / 0.1909 & 0.8361 / 0.1226 \\
  & MSVDD       & 5 & 0.9816 / 0.0082 & 0.9992 / 0.0007 & \textbf{0.9993 / 0.0006} \\
  & \texttt{clusterSVDD} & 10 & 0.9732 / 0.0111 & 0.9401 / 0.0241 & 0.7507 / 0.1259 \\
  & MSVDD       & 10 & 0.8923 / 0.0537 & 0.9686 / 0.0079 & \textbf{0.9889 / 0.0133} \\
  & \texttt{clusterSVDD} & 15 & 0.8489 / 0.0832 & 0.7422 / 0.1161 & 0.6329 / 0.1970 \\
  & MSVDD       & 15 & 0.7707 / 0.1293 & 0.8706 / 0.1201 & \textbf{0.9488 / 0.0338} \\
  & \texttt{clusterSVDD} & 20 & 0.8083 / 0.0285 & 0.7051 / 0.0606 & 0.6030 / 0.1505 \\
  & MSVDD       & 20 & 0.7529 / 0.0157 & 0.8607 / 0.0509 & \textbf{0.9201 / 0.0711} \\
\hline
\end{tabular}
}
}
}
\caption{Mean and standard deviation of AUC-ROC values obtained after cross-validation for instances with real datasets.}
\label{tab:results_real_DB}
\end{table}

\subsection{Outlier Detection and Optimality}\label{sec:Optim_vs_metrics}

One of the advantages of our mathematical optimization model is the solution procedure that is efficiently  implemented in the available optimization solvers to obtain its optimal solution via the construction of a branch-and-bound tree. Then, in the process of optimally solving the problem, a pool of feasible solutions is generated, and used to bound the solutions and certify the optimality of the final one.  Although, in principle, there is not a direct relation between the optimal solution with the training sample and the evaluation of the subsequent classification rule in test samples, the optimal solution is preferred because its structural properties. However, each time a feasible solution is obtained in the branch-and-bound tree, even if it is not certified to be optimal, produces a classification rule that can be evaluated in the test sample, and in case its performance is the best with respect to some metric, one may prefer to output this outlier detection rule instead of the one provides by an optimal solution.

To obtain this pool of feasible solutions, a callback has been implemented within the solution procedure, allowing us to identify and store these solutions. For each incumbent solution, the decision rule is determined, and its AUC-ROC score is calculated for the test sample. Once the optimal solution is obtained, the $gap$ between the optimal and each incumbent solution is calculated, defined as:

\[
gap=\dfrac{Z_{\text{incumbent}} - Z_{\text{optimal}}}{Z_{\text{incumbent}}},
\]
where $Z_{\text{incumbent}}$ and $Z_{\text{optimal}}$ represent the objective values for the incumbent and optimal solutions, respectively. Figure \ref{f:prim_met-vs-gap} illustrates an example of how the accuracy of the classification rule obtained from each incumbent solution varies with respect to the calculated gap.

Tracking the quality of incumbent solutions, measured via AUC-ROC, throughout the branch-and-bound process reveals a relevant insight: although optimality does not necessarily guarantee the best classification performance, reaching the global optimum often results in improved discriminative power for our MSVDD-based decision rule in outlier detection. In particular, the optimal solution frequently achieves a level of separation between regular and anomalous instances that is close to ideal, especially when compared with early, suboptimal solutions.

While a smaller optimality gap does not systematically translate into a higher AUC-ROC, for example, the figure shows a case where a solution with a gap of 0.6 outperforms one with a gap of 0.4, our observations suggest that the most reliable and effective detection rules tend to emerge near or at the global optimum. This is particularly relevant given that the training sample, which determines the optimization problem, and the test sample, used for evaluation, are distinct, and that the evaluation metric (AUC-ROC) differs from the optimization objective.

Overall, these results support the usefulness of exact mathematical optimization techniques in the design of interpretable and effective machine learning models. While not universally superior in all scenarios, optimization-based methods provide a principled and valuable framework for achieving high-quality decision rules.

\begin{figure}[h!]
		\begin{center}
            \includegraphics[scale=0.6]{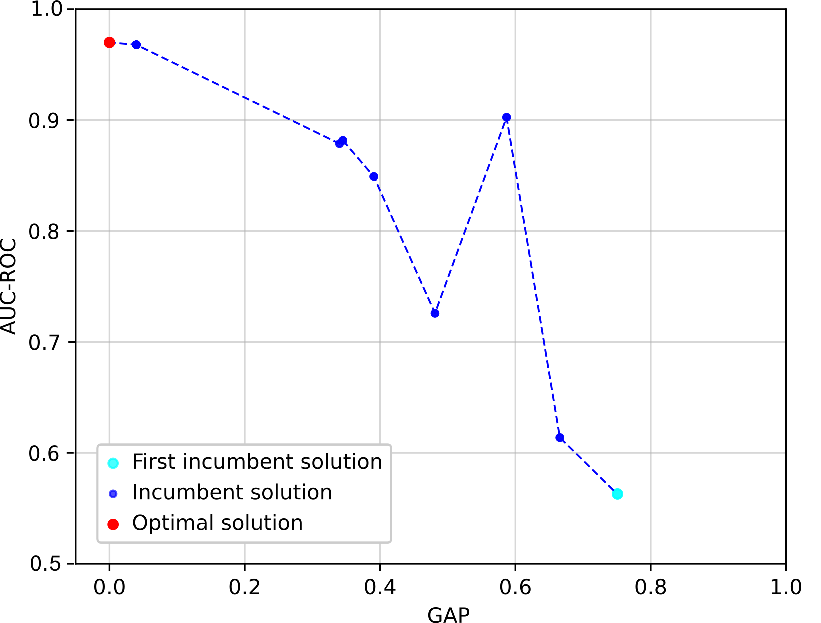}
            \label{f:prim_met-vs-gap_scr_apr}
			\caption{AUC-ROC calculated in every incumbent solution found during the ramification using ($\text{MSVDD}$) primal version ($k=3,\ C=0.8$, \%anom=0.1).
            }
    \label{f:prim_met-vs-gap}
		\end{center}
\end{figure}

 \section{Conclusions and Further Research}\label{sec:conc} 

In this paper, we introduced the Multisphere Support Vector Data Description (MSVDD) problem as a natural extension of the classical Support Vector Data Description (SVDD) model proposed by \cite{tax2004support}, specifically tailored to handle multimodal data distributions. We first developed a primal mathematical optimization formulation and analyzed its structure to derive a dual reformulation. This dual perspective enabled the application of the kernel trick, allowing for the construction of more sophisticated, non-linear decision boundaries that better capture the complex structures often present in real-world datasets.

To the best of our knowledge, this is the first exact optimization-based methodology proposed for the MSVDD problem, in contrast to the heuristic strategies found in the existing literature. Although our method is more computationally intensive, our extensive computational experiments on both synthetic and real datasets show that it significantly improves performance and interpretability, while also providing a sound mathematical foundation for multimodal outlier detection.

Future research  will focus on developing more scalable clustering-based solution techniques to enhance computational efficiency. Although heuristic in nature, these methods could offer quality guarantees in terms of worst-case performance bounds relative to the true optimal value. Furthermore, when prior knowledge about the data distribution is available, the problem may be discretized by providing a predefined set of candidate centers (and radii) for the hyperspheres. The goal would then be to select a subset of these candidates. In this context, techniques from the facility location literature could be leveraged to strengthen the formulation and solve the problem more efficiently. This would also serve as a practical approximation method for the original continuous problem.

 \section*{Acknowledgments}
This work was partially supported by Spanish Ministry of Science and Innovation  (project RED2022-134149-T), Agencia Estatal de Investigaci\'on, Spain and ERDF (projects PID2020-114594GB-C21/C22 and MCIN/AEI/ 10.13039/501100011033) and IMAG-Maria de Maeztu (grant CEX2020-001105-M /AEI /10.13039/501100011033), the European Union NextGenerationEU /PRTR (project TED2021-130875B-I00), and Junta de Andaluc\'ia (project C-EXP-139-UGR23).


\end{document}